\documentclass[10pt,a4paper,reqno]{amsart}
\usepackage[utf8]{inputenc}
\usepackage{amsmath}
\usepackage{amsthm}
\usepackage{import}
\usepackage{graphicx,caption,subcaption}
\usepackage{amsfonts}
\usepackage{amssymb}
\usepackage{palatino}
\usepackage{color}
\usepackage{comment}
\usepackage[left=2.5cm,right=2.5cm,top=2cm,bottom=2cm]{geometry}

\newcommand{\R}{\mathbb{R}}

\newcommand{\N}{\mathbb{N}}
\newcommand{\Z}{\mathbb{Z}}
\newcommand{\cB}{\mathcal{B}}

\newcommand{\cD}{\mathcal{D}}

\newcommand{\cU}{\mathcal{U}}

\newcommand{\bes}{\begin{equation*}}
\newcommand{\ees}{\end{equation*}}
\newcommand{\beas}{\begin{eqnarray*}}
\newcommand{\eeas}{\end{eqnarray*}}
\newcommand{\bea}{\begin{eqnarray}}
\newcommand{\eea}{\end{eqnarray}}
\newcommand{\be}{\begin{equation}}
\newcommand{\ee}{\end{equation}}

\newcommand{\Om}{\Omega}

\newcommand{\bt}{\mathbf{t}}
\newcommand{\bT}{\mathbf{T}}

\newcommand{\Int}{[0,1]}

\newcommand{\DN}{\Delta_n}

\newcommand{\DINSR}{\Delta^{s,r}_{i,k}}
\newcommand{\CP}{P}
\newcommand{\CPBR}{P^{xy}_{su}}

\newcommand{\ECPSX}{E_{P^x_s}}
\newcommand{\RINV}{\Xi_{\ell}}
\newcommand{\tmin}{t^-}

\newcommand{\jbr}{\ell^{x,y}_{s,u}}
\newcommand{\CPBRU}{P^{xy}}
\newcommand{\ECPBR}{E_{\CPBR}}
\newcommand{\ECPBRU}{E_{\CPBRU}}
\newcommand{\BINP}{\cB_{n,p}}
\newcommand{\BINTL}{\cB_{y-x, \pi_{\lambda}(t)}}
\newcommand{\ve}{\varepsilon}

\newtheorem{prop}{Proposition}[section]
\newtheorem{theorem}{Theorem}[section]
\newtheorem{lemma}{Lemma}[section]

\newtheorem{remark}{Remark}[section]

\newtheorem{cor}{Corollary}[section]
\title{Bridges of Markov counting processes: quantitative estimates}
\author{Giovanni Conforti}
\address{Universit\"at Leipzig, Fakult\"at f\"ur Mathematik und Informatik, Augustusplatz 10, 04109 Leipzig, Germany}
\email{giovanniconfort@gmail.com}
\date{November 23,2015}

\begin{document}
\maketitle
\tableofcontents

\begin{abstract}
In this paper we investigate the behavior of the bridges of a Markov counting process in several directions. 
We first characterize convexity(concavity) in time of the mean value in terms of lower (upper) bounds on the so called \textit{reciprocal characteristics}. This result gives a natural criterion to determine whether bridges are "lazy" or "hurried".
Under the hypothesis of global bounds on the reciprocal characteristics we prove sharp estimates for the marginal distributions and a comparison theorem for the jump times. When the height of the bridge tends to infinity we show the convergence to a deterministic curve, after a proper rescaling.
\end{abstract}

\section*{Introduction}
Counting processes are among the simplest stochastic processes but still feature a very broad set of applications in social sciences and engineering, see \cite{AND96},\cite{Bre13Markov},\cite{FLEM}. Moreover, some of the prototypes of jump processes, such as the Poisson processes, fall into this class, and the development of a stochastic calculus adjusted to these (and more general) pure jump processes led to important contributions, see \cite{Bre75},\cite{Jac75},\cite{Wat64}. In this paper we give a new criterion which allows to quantify  the behavior of the \textit{bridges} of a Markov counting process.
One of the most popular interpretations is that counting processes model the number of people arriving in a line.  Using this interpretation, our aim is to find robust  and easy-to-check conditions on the intensity of the counting process under which we can pull out quantitative information on the conditional distribution of the $i$-th arrival time, or on how many customers have arrived by time $t$, given that $n$ customers arrived in a day.\\
Despite the simplicity of the non pinned dynamics, bridges tend to be more complicated to study because their jump intensity is time-dependent and not explicit, but depends on the solution to a Kolmogorov-type equation, which is in general not available. Moreover, jump intensities and bridges are not in one-to one correspondence, in the sense that different counting processes give raise to the same family of bridges. A well known example of this is the fact that all Poisson processes, independently from their jump intensity, have identical bridges. Good estimates have to take into account only those features of the jump intensity which play an effective role in the construction of bridges. \\
Our  answers are  inspired by the theory of reciprocal processes, which provides the main computational tool: the so called \textit{reciprocal characteristics}. They should be thought at some invariants associated with bridges. Reciprocal characteristics have been first discovered for diffusions in \cite{Kre88} and are by now available for a much larger class of processes, including several classes of jump processes, see e.g.\cite{CL15}. However, to the best of our knowledge, it is for the first time in this note that such invariants are used to make quantitative estimates on the behavior of bridges.  \\
As a by product of our analysis we also obtain a natural way of defining when a counting process has \textit{lazy} bridges, meaning that "most customers tend to arrive at the end of the day".  \\
It is interesting to remark that, since counting processes are among the building blocks for general queuing models, our results might open the way for a detailed quantitative study of queuing model under partial information, such as the final state of the queue.\\
The paper is organized as follows: in Section 1 we recall some basics about Markov counting processes and recall some results which we are going to use later on. In Section 2 we present our main results,Theorem \ref{convbr}, Theorem \ref{jumptimesestimate}, and Theorem \ref{thm:lln}. In Section 3 we make their proofs.

\section{Markov Counting processes and their bridges}


The sample  path space $\Omega$ of the counting  processes consists of all c\`adl\`ag step functions with finitely many  jumps with amplitude $+1$ and an initial value in $\N$, from which we exclude paths which jumps either in $t=0$ or $t=1$. The space of probability measures over $\Omega$ is denoted $\mathcal{P}(\Omega)$.
Any path $ \omega\in\Om$ is described by the collection $(x;t_1,\dots,t_n)$ of its initial position $x\in\N$ and its $n= \omega_1- \omega_0$ instants of jumps $0<t_1<\cdots<t_n<1.$ We denote $T_i( \omega):=t_i$ the $i$-th instant of jump of $ \omega.$
The canonical process is denoted by $X=(X_t) _{ 0\le t\le 1}$, and we call counting process any probability law on $\Omega$
Any counting process $P$ admits an almost surely unique increasing process $A:[0,1]\times \Omega \rightarrow \R_+$ such that $P(A(0)= 0)=1$ and 
$	t \mapsto X_t - X_0 - A(t)
$  is a local $P$-martingale,
(see Jacod \cite[Thm.\ 2.1]{Jac75}, for instance). $A$  characterizes the dynamics of $P$, and it is called the compensator.
When the compensator is absolutely continuous, 
we call its derivative  the {\it intensity} of $P$. 
We consider \textit{Markov} counting processes, meaning that the intensity depends only on time and the current state. We assume that the  jump intensity $\ell: \Int \times \N \rightarrow \R_{+}$ satisfies the following:
\begin{enumerate}
\item[i)]  For all $z \in \N$, $t \mapsto \ell(t,z)$ is continuously differentiable in $[0,1]$.
\item[ii)] The function $ \Int \in \N \ni t,z \in \Int$, $z \mapsto \ell(t,z+1)-\ell(t,z)$ is globally upper and lower bounded.
\item[iii)]  For all $z \in \N$, $t \in \Int$,  $\ell(t,z)>0$.
\end{enumerate}
From now on, a reference counting process $P$ with intensity satisfying i),ii),iii) and whose initial measure has full support is fixed. 
In the rest of the paper whenever we refer to the intensity of a Markov counting process, we always assume that it satisfies i),ii), and iii) even if we do not specify it. i),ii) ensure existence and non-explosion of the process, (for the non explosion part, see \cite[Theorem 2.3.2]{NORRIS}), while iii) ensures that bridges are well defined between any pair of states $x,y$ with $x \leq y$.
 We denote by $\CPBR$ the $xy$ bridge between $s$ and $u$.
\bes
\CPBR(\cdot):= \CP( \cdot \, \vert X_s = x, X_{u}= y)
\ees
When $s=0,u=1$ we simply write $\CPBRU$. We will call $u-s$ the length of the bridge whereas $y-x$ is the height of the bridge. Moreover, for any $s \in [0,1]$, $x\in \N$ we define 
\bes
P^x_s(\cdot) = P(\cdot \, \vert X_s =x)
\ees 
Expectation under $P$ is denoted $E_{P}$. $\ECPBR$ and $\ECPBRU$,$E_{P^x_s}$ are defined analogously.
\subsection{Duality formula for the bridge of a Markov counting process}\label{subs:duality}
The reciprocal characteristic is a map $\RINV: \N \times \Int \rightarrow \R $ derived from the intensity of $P$, which encodes all the necessary and sufficient information on $\ell$ to construct the family of bridges $\{ \CPBR \}_{x \leq y \in \N, s\leq u}$. It is defined as\footnote{In the definition of the reciprocal characteristic we allowed for $\ell(t,z)=0$, which was excluded by our hypothesis iii). This is because in the proof of Theorem \ref{convbr} we will consider the reciprocal characteristics of the intensities of the bridges of $P$, which fail to satisfy iii). However, the reference intensity $\ell$ always satisfies this condition}:
\bes
\RINV(t,z):=
\begin{cases}
\partial_t \log \ell(t,z) + \ell(t,z+1)-\ell(t,z), \quad & \text{if $\ell(t,z) \neq 0$} \\
0, \quad & \text{if $\ell(t,z) = 0$}
\end{cases}
\ees
A short-time probabilistic interpretation of $\RINV$ can be found in a larger generality in \cite[Theorem 2.6]{CL15}: it describes the asymptotic distribution of jump times in a bridge whose time length is very short. It is part of this paper to extract quantitative bounds in a \textit{non} asymptotic time scale. 
At this point, it is worth recalling a duality formula proved in \cite{CLMR}, building on earlier work in  \cite{carlen88}, characterizing the bridges of $P$, as we are going to use it in the proofs.
 To streamline the presentation, we omit some minor details. For precise statements we refer to 
\cite[Theorem 2.12]{CLMR}. The duality formula is an integration by parts formula on the  path space $\Omega$ which puts in relation a derivative operator with a stochastic integral operator.
The derivative operator acts on test functions $\Phi:\Om\to\R$ of the form:
\bes
\Phi =\varphi\big(X_0;T_1,\dots, T_m\big) , \quad m \geq 1, \varphi:\N \times [0,1]^m \to\R 
\ees
where $\varphi$ is such that for all $x\in\Z$, the partial functions $\varphi(x;\cdot)$ are  ${\mathcal
C} ^{ \infty}$-differentiable.
The directions of differentiation are given by the set of periodic functions \bes \cU = \left\{ u \in \mathcal{C}^1(\Int;\R): u(0)=u(1)=0 \right\}\ees 
The derivative operator acts as follows on a test function $\Phi$:
\be\label{definizionederivata}
\cD_u\Phi	= - \sum_{j=1}^m \partial_{t_j} \varphi(X_0;T_1,\dots, T_m) u(T_j) 
\ee
After recalling that the stochastic integral  $\int_{0}^{1} f(t,X_{\tmin}) d X_{t}$ is, as usual, $\sum_{T_i <1} f(T_i,X_{T^-_i})$
we can state the formula.
\begin{theorem}\label{thm:derivcharacrec}
Let $P$ be a counting process of intensity $\ell$. Then $P^{xy}$ is the only element of $\mathcal{P}(\Om)$ such that $P^{xy}(X_0=x,X_1=y)=1$ and
\begin{equation} \label{dualityformula}
E_{P^{xy}} (\mathcal{D}_{u} \Phi)
=
E_{P^{xy} }\Big(\Phi\int _{0}^1 \big[\dot{u}(t)+ \Xi_{\ell}(t,X_{t^-})u(t)\big]\,dX_t \Big)
\end{equation}
holds for any test function $\Phi$ and any 
$u \in \mathcal{U}$.
\end{theorem}

\section{Quantitative estimates }\label{sec:quantes}
\subsection{Convexity and concavity in time of the mean value}

Our first result is the equivalence between global bounds on $\RINV$ and the convexity/concavity of the mean value, as a function of time.
\begin{theorem}\label{convbr}
Let $P$ be a counting process. If
\bes
\inf_{t \in \Int, z \in \N} \RINV(t,z) \geq 0 ,
\ees
 then for all $0 \leq s \leq u \leq 1$ and for all $x \leq y \in \N$ the function $ [s,u] \ni t \mapsto \ECPBR(X_t)$ is convex.
Conversely, if
\bes
\sup_{t \in \Int, z \in \N} \RINV(t,z) \leq 0 ,
\ees
 then for all $0 \leq s \leq u \leq 1$, and for all $x \leq y \in \N$ the function $ [s,u] \ni t \mapsto \ECPBR(X_t)$ is concave.
\end{theorem}
Thanks to this result, it makes sense to say that the bridges of a counting process are "lazy" when $\RINV$ is a non negative function. Indeed, using the interpretation that $X_t$ counts the number of arrivals of customers, Theorem \ref{convbr} states that, on average, customers arrive at a slow rate at the beginning of the day, and this rate always increases over time in order to match the information that $y-x$ customers arrived in the whole day: hence arrivals are concentrated towards the end of the day. \\
 In particular, at time $t$, one expects to have observed less than $n \ t$ arrivals, ad the curve $ \Int \ni t \mapsto E_{P^{xy}}(X_t)$ lies below the straight line between $x$ and $y$. \\
    On the contrary,  if $\RINV$ is non-positive we can say that bridges are "hurried". Customers arrive very quickly at the beginning, and the graph of $\Int \ni t \mapsto E_{P^{xy}}(X_t)$ lies above the straight line between $x$ and $y$. The arrival of customers has to slow down as time grows (concavity), due to the constraint that $y-x$ customers, and not more, have to be arrived at the end of the day. We refer to Figure \ref{fig-1} for an illustration of these concepts. \\

\subsection{Estimates for the marginals}

Next result turns the qualitative statement of Theorem \ref{convbr} into  a quantitative one. We start from the observation that there are three main families of Markov counting processes such that $\RINV$ is a time-space constant function:
\begin{itemize}
\item The Poisson processes, i.e. $\ell(t,z) \equiv \alpha>0 $.  Regardless of the value of $\alpha$, their bridges are neither lazy nor hurried. The average rate which customers arrive is constantly equal to $y-x$.
\item Markov counting processes with time homogeneous and space linear rates, i.e. $\ell(t,z) = \lambda z + \alpha$ for some $\lambda, \alpha>0$.
\item The space-homogeneous counting processes with exponential rate, i.e. $\ell(t,z) = \alpha \exp(\lambda t)$, for some $\lambda \in \R, \alpha>0$. Their bridges may be lazy or hurried, depending on whether $\lambda$ is positive or not.
\end{itemize}

For this kind of models all sort of explicit computations are possible: we will use them as a benchmark in the next theorem. As we will see in Lemma \ref{explicitcomputation}, the marginal at time $t$ of the $xy$ bridge of a process with $\RINV \equiv \lambda $ is a binomial distribution whose parameters are determined by $\lambda,x,y$ and $t$. This is a generalization of the well known fact that marginals of Poisson bridges are binomial distributions.

We make use of this explicit computation to provide estimates for the case when the marginal distributions are not known in closed form. We shall see that if $\RINV$ admits a non negative lower bound, say $\lambda$, then any  bridge of $P$ has marginals which have lighter tails than the marginals of the same bridge built from a process satisfying $\RINV \equiv \lambda$. Therefore we prove that the tails are lighter than the tails of a certain binomial law. 
In this sense, Theorem \ref{jumptimesestimate} strengthens the definition of lazy bridge, giving a parameter to measure quantitatively such laziness, and creating a partial order in the family of processes with lazy bridges. All statements can be reversed when lower bounds become upper bounds.
 
To state the Theorem, some notation is needed:
we denote by $\BINP$ the binomial distribution associated with $n$ trials and probability of success $p$.
\bes
\forall \ 0 \leq k \leq n, \quad \quad \cB_{n,p}(k) = \binom{n}{k} p^k (1-p)^{n-k}
\ees
We define for all $\lambda \in \R$ the function $\pi_{\lambda}: \Int \rightarrow \Int$ as
\be\label{pilambda}
\pi_{\lambda}( t) = \frac{\exp(\lambda t )-1}{\exp(\lambda)-1}
\ee
\begin{theorem}\label{jumptimesestimate}
Let $P$ be a counting process of intensity $\ell$, $x \leq y \in \N$.
If  for some $\lambda \in \R$, 
\be\label{eq:densityest}
\inf_{t \in \Int, x \leq z \leq y-1 } \RINV(t,z) \geq \lambda 
\ee
then for  $1\leq i \leq y-x$ and $t \in [0,1]$:
\bes
\CPBRU(X_t  \geq x+i ) \leq \BINTL(\{i,.,y-x\})
\ees

Conversely, if for some $\lambda \in \R$,
\bes
\sup_{t \in \Int, x \leq z \leq y-1} \RINV(t,z) \leq \lambda 
\ees

then for any $1\leq i \leq y-x$ and  $t \in [0,1]$:
\bes
\CPBRU(X_t \geq x+i ) \geq \BINTL(\{i,.,y-x\})
\ees
\end{theorem}

\begin{remark}
 The inequalities we obtain in this Theorem are equalities when $\RINV(t,z) \equiv \lambda $ (see Lemma \ref{explicitcomputation}) , and therefore are sharp.
\end{remark}

A simple consequence of Theorem \ref{jumptimesestimate} is the following mean value estimate.
\begin{cor}\label{meanvalueest}
Let $P$ be a counting process of intensity $\ell$, $x \leq y \in \N$.
If  
\bes
\inf_{t \in \Int, z \in \N} \RINV(t,z) \geq \lambda 
\ees
then for any $x \leq y$ and $t \in [0,1]$:
\bes
E_{P^{xy}}(X_t) \leq x+ (y-x) \frac{\exp(\lambda t)-1}{\exp(\lambda)-1 }
\ees
\end{cor}


\subsection{Law of large numbers}
Here, we prove an asymptotic result under the hypothesis that the $\RINV(t,z)$ converges as $z \rightarrow + \infty$ to $\lambda \in \R$.
We will see  that as $n \rightarrow + \infty$, the bridge of length one between $0$ and $n$, after a suitable rescaling, converges to the deterministic curve $t \mapsto \pi_{\lambda}(t)$ (see Figure \ref{fig-1}).

\begin{theorem}\label{thm:lln}
Let $P$ be a Markov counting process of intensity $\ell$. Assume that for some $\lambda \in  \R$:
\be\label{limitass}
\lim_{z \rightarrow + \infty } \sup_{t \in [0,1]} \vert \RINV(t,z) - \lambda \vert =0
\ee
Then for every $\varepsilon > 0$: 
\bes
\lim_{N \rightarrow + \infty }P^{0N}\left(\sup_{t \in [0,1]} \Big\vert \frac{1}{N} X_t - \pi_{\lambda}(t)  \Big\vert \leq \varepsilon \right) =0
\ees
\end{theorem}
\newpage
  \begin{figure}[ht!]
\centering
 \includegraphics[height=9cm,width=12cm]{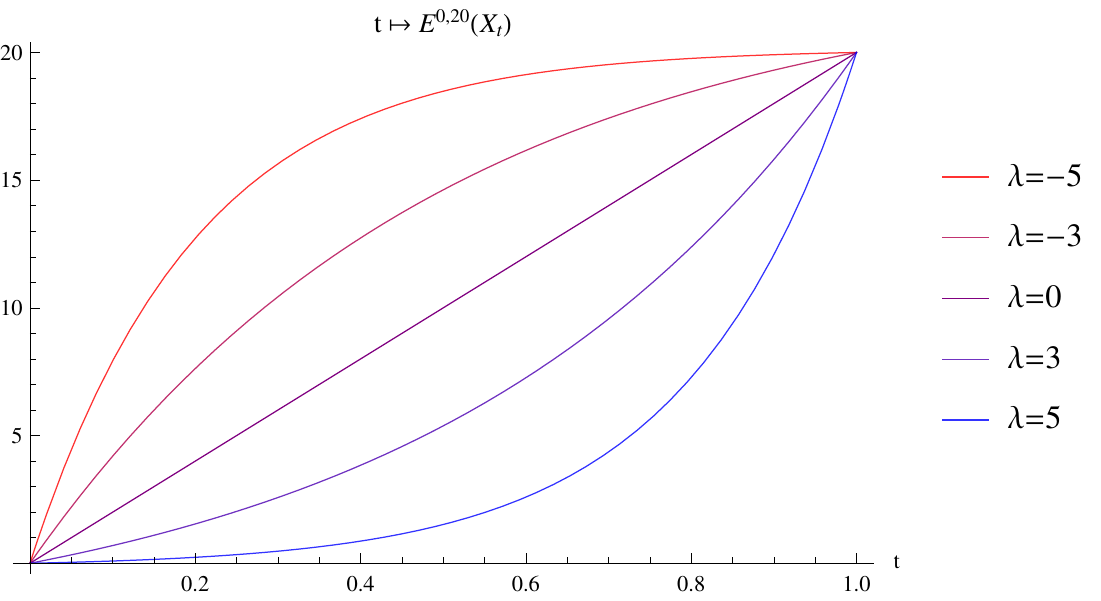}
 \includegraphics[height=9cm,width=12cm]{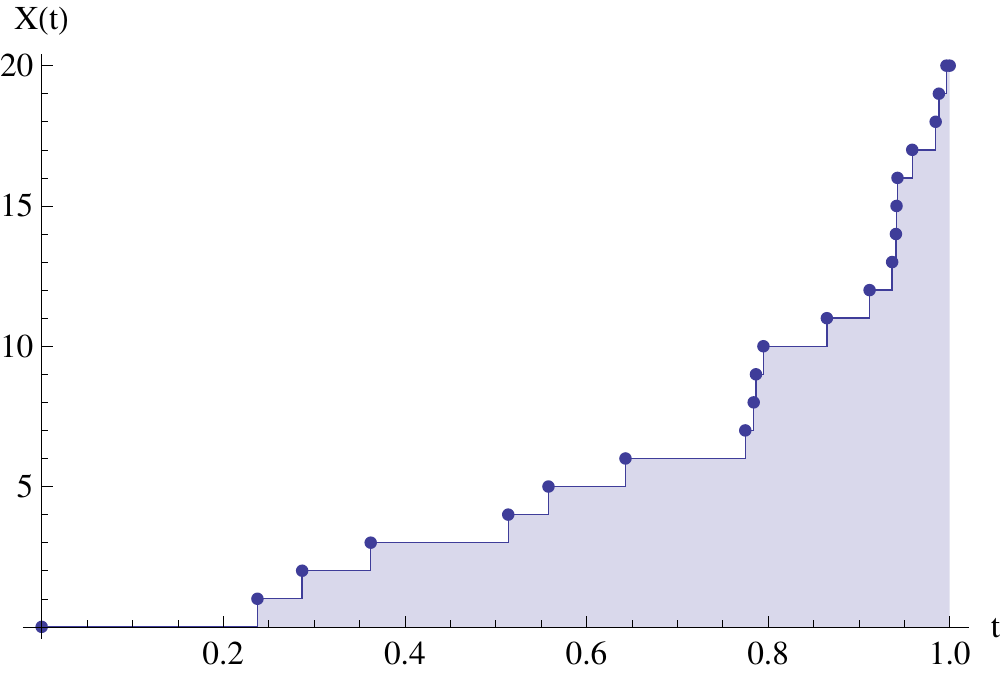}
\caption{Top: We plot for $\lambda \in \{ -5,-3,0,3,5\}$ the function $t \mapsto E_{P^{0,20}}(X_t)$, where the intensity $\ell$ of $P$ satisfies $\RINV \equiv \lambda$. For negative values of $\lambda$ we have a concave function, whereas for positive values a convex one. Moreover, we see that, as a function of $\lambda$, $ E_{P^{0,20}}(X_t)$ is decreasing. Bottom:  A sample path of a lazy bridge $(\lambda=3$): it displays very little activity close to zero: almost all jump times accumulate around $t=1$.}
 \label{fig-1}
\end{figure}
\section{Proofs of the results}\label{subs:proof}
\subsection*{Proof of Theorem \ref{convbr} }\label{subs:convexity}
The proof of Theorem \ref{convbr} builds on the earlier results of \cite{CLMR} and on Lemma \ref{conv}below. There, we prove the statement under $P$ instead of $P^{xy}$. Unless in the rest of the paper, we do not assume $\ell$ to satisfy point iii) in the proof of this theorem. This is because we will eventually  apply it  to a bridge $P^{xy}$, whose intensity does not satisfy this condition.
\begin{lemma}\label{conv}
Let $P$ be a Markov counting process of intensity $\ell$. 
Then 
\bes
\inf_{t \in \Int, z \in \N} \RINV(t,z) \geq 0 
\ees
if and only if for all $s,x \in [0,1] \times \N$ the function $ [s,1]\ni t \mapsto \ECPSX(X_t)$ is convex.\\
Conversely, 
\bes
\sup_{t \in \Int, z \in \N} \RINV(t,z) \leq 0 
\ees
if and only if for all $s,x \in [0,1] \times \N$ the function $ [s,1]\ni t \mapsto \ECPSX(X_t)$ is concave.
\end{lemma}

\begin{proof}
We only prove the convexity statement, the other being analogous. \\
$(\Rightarrow)$ Fix $s \in \Int, x \in \N$ and define $\varphi_t:= \ECPSX(X_t)$ for $t \in [s,1]$.
Differentiating twice we obtain:
\beas
\dot{\varphi}_t &=& \ECPSX(\ell(t,X_{\tmin})) \\
 \ddot{\varphi}_t &=& \ECPSX( \partial_t \ell(t,X_{\tmin})+( \ell(t,X_{\tmin}+1)-\ell(t,X_{\tmin}) ) \ell(t,X_{\tmin})  ) \\
 &=&  \ECPSX( \RINV(t,X_{\tmin}) \ell(t,X_{\tmin}) \mathbf{1}_{\{\ell(t,X_{\tmin}) >0\} } + \partial_t \ell(t,X_{\tmin}) \mathbf{1}_{\{\ell(t,X_{\tmin}) =0 \}})
\eeas
Since $ \ell(s,z)$ is continuously differentiable and for all $t \in [0,1]$ $P(X_{t}-X_{\tmin} \neq 0) = 0$, we have that, for a fixed $t$, $\ell(t,X_{t^-})=0$ a.s. implies that $\partial_t \ell(t,X_{t^-})=0$. Therefore by definition of $\RINV$ we have for all $t$:
\bes \ECPSX(\partial_t \ell(t,X_{\tmin}) \mathbf{1}_{\{\ell(t,X_{\tmin}) =0 \}}) =  \ECPSX(0 \mathbf{1}_{\{\ell(t,X_{\tmin}) =0 \}}) = \ECPSX(\RINV(t,X_{\tmin})\ell(t,X_{\tmin}) \mathbf{1}_{\{\ell(t,X_{\tmin}) =0 \}})
\ees
and therefore:
\bes
 \ddot{\varphi}_t =\ECPSX(\RINV(t,X_{\tmin})\ell(t,X_{\tmin}) )
\ees
If $\inf_{t,z}\RINV(t,z) \geq 0$, we then have $\ddot{\varphi}_t \geq 0$, which gives the convexity. \\
$(\Leftarrow)$ Consider $s \in \Int$,$x \in \N$ such that $\ell(s,x)>0$ (note that if $\ell(s,x)=0$, $\RINV(s,x)=0$ and there is nothing to prove). Then, defining $\varphi$ as above and repeating the same computations we obtain, by evaluating the second derivative at $s$, and exploiting the fact that $X_s=x$ almost surely:
\bes
\ddot{\varphi}_s = \RINV(s,x) \ell(s,x)
\ees 
Since $\ddot{\varphi}_s \geq 0 $ by assumption and $\ell(s,x)>0$, then $\RINV(s,x) \geq 0$. Because of the arbitrary choice of $s$ and $x$ the conclusion follows.
\end{proof}
The proof of Theorem \ref{convbr} becomes now almost straightforward.
\begin{proof}
We exploit the fact $\CPBR$ is also a counting process whose intensity $\jbr (t,z)$ is continuously differentiable on $[0,1) \times \N$ and  satisfies:
\bes
\Xi_{\jbr} (t,z) = \begin{cases} \RINV(t,z), \quad & \text{if $x \leq z \leq y-1 $, $0<t<1$} \\ 0, & \text{ otherwise }  \end{cases}
\ees
This is proven in \cite[Theorem 1.11]{CLMR}. An application of Lemma \ref{conv} gives then the conclusion.
\end{proof}
\subsection*{Proof of Theorem \ref{jumptimesestimate}}\label{subs:proofestimate}
Prior to the proof, we introduce some useful notation which simplifies many computations therein. We call $n:=y-x$ the total number of jumps that the $xy$ bridge makes. Moreover, we adopt the following conventions:
\begin{itemize} 
\item For $(t_1,..,t_n) \in \Int^n$ we adopt the compact notation $\bt$. For any $1 \leq i \leq k \leq n$ the vector $(t_i,..,t_k)$ is denoted $\bt_{i,k}$.

\item For $1 \leq i \leq k \leq n$, $0 \leq s \leq r \leq 1$, we call $\DINSR$ the set :
$$
\DINSR:= \{ (t_{i},...,t_{k}) \in [0,1]^{k-i+1} : s < t_{i} < t_{i+1} <..<t_k <r \}
$$
When $s=0,r=1,i=1,k=n$, we simply write $\Delta_n$ .
\item
For any $1 \leq j \leq n$ we define $ \xi_j:\Int \longrightarrow \R$ as the only primitive of $t \mapsto \RINV(t,x+j-1)$ such that $\xi_j(0)=0$. We then define \bes \xi:\DN \longrightarrow \R, \quad 
\xi(\bt) = \sum_{j=1}^n \xi_j(t_j) \ees 

We also define for all $1\leq i \leq k \leq n$:
\bes
\xi_{i,k}(\bt_{i,k}) = \sum_{j=i}^{k} \xi_j(t_j)
\ees
\end{itemize}

The proof is structured as follows: we first prove two Lemmas, Lemma \ref{bridgedist} Lemma \ref{explicitcomputation}, which contain the two main ingredients of the proof. Then we do the proof of Theorem \ref{jumptimesestimate}, relying on the more technical Lemma \ref{increasingratio2}, which is proven later, together with the simple auxiliary Lemma \ref{increasingratio}.
\begin{lemma}\label{bridgedist}
The distribution of jump times $(T_1,..,T_n)$ under $P^{xy}$ is given by:
\be\label{eq:lawofjtimes}
\forall A \subseteq \Delta_n, \quad \CPBRU(T_1,..,T_n \in A) = \frac{1}{Z_{\ell}} \int_{A} \exp \big(\xi(\bt) \big) d\bt
\ee
\end{lemma}
\begin{proof}
For convenience, we denote the vector $(T_1,..,T_n)$ by $\bT$.
Let us define the probability measure $Q \in \mathcal{P}(\Omega)$ by:
\bes
dQ = \frac{1}{Z_{\ell}} \exp(\xi(\bT) ) dR^{xy}
\ees
where $R^{xy}$ is the Poisson bridge from $x$ to $y$. Thanks to the duality formula \eqref{dualityformula}, the Poisson bridge is characterized by:
\be\label{poisson duality}
\forall \Phi, u \in \cU, \quad  E_{R^{xy}} (\mathcal{D}_{u} \Phi)
=
E_{R^{xy} }\Big(\Phi\int _{0}^1 \dot{u}(t)\,dX_t\Big)
\ee
Consider now a test function $\Phi$, using the duality formula \eqref{dualityformula} and the fact that $\cD_u$ is a true derivative operator satisfying Leibniz's product rule:
\beas
E_{Q}(\cD_u \Phi)& =& E_{R^{xy}}\Big( \exp(\xi(\bT)) \, \cD_{u}\Phi \Big) \\
							&=& E_{R^{xy}}\Big( \cD_{u} \big[ \exp(\xi(\bT)) \, \Phi \big] \Big) -E_{R^{xy}}\Big(\Phi  \, \cD_{u} \big[\exp(\xi(\bT) ) \big] \Big) \\
							&\underbrace{=}_{\text{eq.} \, \eqref{poisson duality} + \text{eq.} \eqref{definizionederivata} }&	E_{R^{xy}}\Big( \exp(\xi(\bT)) \Phi \int_{0}^{1} \dot{u}(t) dX_t \Big) + E_{R^{xy}} \Big(\Phi \exp(\xi(\bT)\sum_{j=1}^n \underbrace{\partial_{t_j}  \xi(\bT)}_{= \RINV(T_j,X_{T^-_{j}})} u(T_j))  \Big) \\
&=&E_{Q} \Big(\Phi \int_{0}^1\big[ \dot{u}(t) + \RINV(t,X_{\tmin})u(t) \big] dX_t \Big)		\eeas

An application of Theorem  \ref{thm:derivcharacrec} allows to deduce that $Q=P^{xy}$. Therefore, the density of $P^{xy} \circ \bT^{-1} $  with respect to $R^{xy} \circ \bT^{-1} $ is proportional to $\exp(\xi(\bt))$. It is well known that $R^{xy} \circ \bT^{-1}$ is the normalized Lebesgue measure on $\Delta_n$. The conclusion then follows.
\end{proof}
Let us compute explicitly the distribution of the jump times of a  bridge of a counting process such that $\RINV$ is  constantly equal to $\lambda$.

\begin{lemma}\label{explicitcomputation}
Let $P$ be a counting process of intensity $\ell$, $ x\leq y \in \N$. Let $\ell$ be such that for some $\lambda \in \R$ :
\bes
\forall t \in \Int , \, x \leq z \leq y-1, \quad \RINV(t,z) = \lambda
\ees
Then, for all $t \in \Int, 1 \leq i \leq y-x $ we have that the law of $X_t -x $ is $\mathcal{B}_{y-x, \pi_{\lambda}(t)}$, where $\pi_{\lambda}(t)$ is defined by \eqref{pilambda}.
\end{lemma}
\begin{proof}
As a direct consequence of Lemma \ref{bridgedist}
combined with the current hypothesis we have:
\be\label{eq-10}
P^{xy}(X_t \geq x+i)P^{xy}(T_i \leq t) = \frac{1}{Z_{\lambda}} \int_{\Delta_{n} \cap  \{ t_i \leq t \}} \exp( \lambda \sum_{j=1}^{n} t_j)d\bt
\ee
where $Z_{\lambda}$ is a normalization constant.
Thanks to the invariance of the integrand under permutations we have:
\bea\label{symmetrizedintegral3}
\frac{1}{Z_{\lambda}} \int_{\Delta_n \cap \{ t_i \leq t \}} \exp \big(\lambda \sum_{j=1}^n t_j \big) d\bt 
= \frac{1}{(\exp(\lambda) -1)^n} \int_{A}   \exp(\lambda \sum_{j=1}^n t_j) d\bt
\eea
with
\bes\label{symmetrizedintegral4}
A:= \bigcup_{\sigma \in S_n} \{t: t_{\sigma(1)}< ..<t_{\sigma(n)}, t_{\sigma(i)} \leq t \} 
\ees
where $S_n$ is the symmetric group with $n$ elements.
We can rewrite $A$ in a convenient way, using some simple combinatorial arguments. We have: 
\bes
A:= \{ \bt \in [0,1]^n s.t. |\{j: t_j \leq t \} | \geq i \}
\ees
But then the right hand side of \eqref{symmetrizedintegral3} is nothing but the probability that at least $i$ among $n$ i.i.d. random variables are less than $t$, each variable being supported on $[0,1]$ and having a density proportional to $\exp(\lambda s)$. The conclusion then follows after a simple calculation.
\end{proof}
We can now prove Theorem \ref{jumptimesestimate}.
\begin{proof}
Because of Lemma \ref{bridgedist} and Lemma \ref{explicitcomputation}, and the elementary fact that $\{X_t \geq x+ i \} = \{ T_i \leq t \}$, all what we need to do is to prove that for all $1 \leq i \leq n, 0 \leq t \leq 1$:
\bea\label{eq1}
\frac{1}{Z_{\ell}} \int_{\bt \in \DN \cap \{ t_i \leq t \}} \exp(\xi(\bt)) d\bt \leq \frac{1}{Z_{\lambda}} \int_{\bt \in \DN \cap \{ t_i \leq t \}} \exp \big( \lambda \sum_{j=1}^{n} t_j \big) d\bt
\eea
where $Z_{\ell},Z_{\lambda}$  are normalization constants.

To this aim,  we define the function $\rho:[0,1] \rightarrow \R_+$ by:
\bes \rho(r):= \frac{ \int_{\bt \in \DN \cap \{ t_i \leq r \}} \exp(\xi(\bt)) d\bt}{ \int_{\bt \in \DN \cap \{ t_i \leq r \}} \exp \big( \lambda \sum_{j=1}^{n} t_j \big) d\bt} \ees
Observing that $\rho(1)=\frac{Z_{\ell}}{Z_{\lambda}}$,  \eqref{eq1} is equivalent to $\rho(t) \leq \rho(1)$, and therefore it is sufficient to prove that $\rho$ is non decreasing. 
For this, we rewrite $\rho(t)$ in a convenient way, by conditioning on the value of $t_i$:
\be\label{eq:rho}
\rho(r)= \frac{\int_{0}^{r} p(s) ds}{\int_{0}^{r} q(s) ds}
\ee
with:
\beas
p(s)&:=& \int_{\bt \in \DN \cap \{  t_i = s \} }\exp \big(\xi(\bt)\big) d\bt_{1,i-1} d\bt_{i,n} \\
&=& \exp(\xi_i(s)) \int_{\Delta^{0,s}_{1,i-1}}\exp \big( \xi_{1,i-1}(\bt_{1,i-1}) \big)d\bt_{1,i-1}
\int_{\Delta^{s,1}_{i+1,n}}\exp \big( \xi_{i+1,n}(\bt_{i+1,n}) \big)d\bt_{i+1,n} \\
q(s)&:=&\int_{\bt \in \DN \cap \{t_i = s\}}\exp \big(\lambda \sum_{j=1}^n t_j \big) d\bt_{1,i-1}d\bt_{i+1,n} \\
&=& \exp(\lambda s) \int_{\Delta^{0,s}_{1,i-1}}\exp \big(\lambda\sum_{j=1}^{i-1} t_j \big)d\bt_{1,i-1}
\int_{\Delta^{s,1}_{i+1,n} }\exp \big(\lambda \sum_{j=i+1}^{n} t_j \big)d\bt_{i+1,n} 
\eeas
Let us first show that $p/q(s)$ is non decreasing. Lemma \ref{increasingratio2} ensures that both
\bes s \mapsto \frac{\int_{\Delta^{0,s}_{1,i-1} } \exp(\xi_{1,i-1}(\bt_{1,i-1}) ) d\bt_{1,i-1} }{ \int_{\Delta^{0,s}_{1,i-1}} \exp(\lambda \sum_{j=1}^{i-1} ) d\bt_{1,i-1} } \quad
\text{and}
 \quad  s \mapsto \frac{\int_{\Delta^{s,1}_{i+1,n} } \exp(\xi_{i+1,n}(\bt_{i+1,n}) ) d\bt_{i+1,n} }{ \int_{\Delta^{s,1}_{i+1,n}} \exp(\lambda \sum_{j=i+1}^{n}t_j ) d\bt_{i+1,n} }
\ees
are non decreasing in $s$. Moreover, $\exp(\xi_i(s) - \lambda s)$ is also increasing in $s$ since $\dot{\xi}_{i}(\cdot) = \RINV(\cdot,x+i-1) \geq \lambda$ by hypothesis. We have thus shown that $p/q(s)$ is non decreasing. By Lemma \ref{increasingratio} and \eqref{eq:rho} we get that $\rho$ is non decreasing too, and the conclusion follows.
\end{proof}

\begin{lemma}\label{increasingratio2}
The functions
 \bes s \mapsto \frac{\int_{\Delta^{0,s}_{1,i-1} } \exp(\xi_{1,i-1}(\bt_{1,i-1}) ) d\bt_{1,i-1} }{ \int_{\Delta^{0,s}_{1,i-1}} \exp(\lambda \sum_{j=1}^{i-1} ) d\bt_{1,i-1} }
\ees

and

\bes
s \mapsto \frac{\int_{\Delta^{s,1}_{i+1,n} } \exp(\xi_{i+1,n}(\bt_{i+1,n}) ) d\bt_{i+1,n} }{ \int_{\Delta^{s,1}_{i+1,n}} \exp(\lambda \sum_{j=i+1}^{n}t_j ) d\bt_{i+1,n} }
\ees
are non decreasing  in $s$.
\end{lemma}

\begin{proof}
We only prove the first statement, the other case being completely analogous.
We work by induction on $i$. For the case $i=2$, we have to show that 
\bes
s \mapsto \frac{\int_{0}^{s} \exp(\xi_1(t_{1})  ) dt_1}{\int_{0}^{s} \exp(\lambda t_1) dt_1 }
\ees
is non decreasing. The assumption of Theorem \ref{jumptimesestimate} ensures that $\xi_1(t_1) - \lambda t_1$ is non decreasing, since $\lambda \leq \RINV(\cdot, x)=\dot{\xi}_1(\cdot) $ by hypothesis.
But then we conclude by applying Lemma 3.5. For the inductive step we rewrite 
\bes
\int_{\Delta^{0,s}_{1,i-1}} \exp(\xi_{1,i-1}(\bt_{1,i-1}) ) d \bt_{1,i-1} = \int_{0}^{s}\exp(\xi_{i-1}(t_{i-1})) \left\{ \int_{\Delta^{0,t_{i-1}}_{1,i-2}} \exp(\xi_{1,i-2}(\bt_{1,i-2}) ) d \bt_{1,i-2} \right\} d t_{i-1}
\ees
and similarly
\bes
\int_{\Delta^{0,s}_{1,i-1}} \exp( \lambda \sum_{j=1}^{i-1}t_j) d \bt_{1,i-1} = \int_{0}^{s}\exp(\lambda t_{i-1}) \left\{ \int_{\Delta^{0,t_{i-1}}_{1,i-2}} \exp(\sum_{j=1}^{i-2} t_{j}  ) d \bt_{1,i-2} \right\} d t_{i-1}
\ees
Using the inductive hypothesis, we have that 
\bes
t_{i-1} \mapsto \frac{\int_{\Delta^{0,t_{i-1}}_{1,i-2}} \exp(\xi_{1,i-2}(\bt_{1,i-2}) ) d \bt_{1,i-2} }{ \int_{\Delta^{0,t_{i-2}}_{1,i-1}} \exp(\sum_{j=1}^{i-2} t_{j}  ) d \bt_{1,i-2}}
\ees
is increasing. Moreover, $t_{i-1} \mapsto \frac{\exp(\xi_{i-1}(t_{i-1}) )}{\exp(\lambda t_{i-1})}$ is also increasing, since by assumption $   \lambda \leq \RINV(\cdot,x+i-2 )= \dot{\xi}_{i-1}(\cdot)$. The conclusion follows by Lemma \ref{increasingratio}.
\end{proof}
\begin{lemma}\label{increasingratio}
Let $p,q$ be continuously differentiable strictly positive  functions such that $\frac{p}{q}(s)$ is increasing. Then 
\be\label{eq:increasingratio}
t \mapsto \frac{\int_{0}^{t} p(s) ds}{\int_{0}^{t} q(s) ds} 
\ee
is also increasing.
\end{lemma}
\begin{proof}
We differentiate \ref{eq:increasingratio} in $t$. We obtain:
\bes
\frac{1}{\big(\int^{t}_{0}q(s)ds \big)^2}\Big[p(t) \int^{t}_{0}q(s)ds - q(t) \int_{0}^{t}p(s)ds\Big]
\ees
We observe that the conclusion follows if we can prove that 
$$ \forall  s \leq t \quad p(t)q(s) \geq q(t)p(s)$$
Since $s\leq t$, our hypothesis ensures this.
\end{proof}
Let us prove Corollary \ref{meanvalueest}
\begin{proof}
We simply write:
\bes
E_{P^{xy}}(X_t) = x+\sum_{i=1}^n P^{xy}(T_i \leq t) \underbrace{\leq}_{\text{Th.} \,\ref{jumptimesestimate}} \sum_{i=1}^n \cB_{n,\pi_{\lambda}(t)}(T_i \leq t) =x+ n \frac{\exp(\lambda t)-1}{\exp(\lambda)-1}.
\ees
Recalling that $n-y-x $ the conclusion follows.
\end{proof}
\subsection*{Proof of Theorem \ref{thm:lln}}

Theorem \ref{jumptimesestimate}  can be adapted to the case when we consider bridges of non-unitary length. The proof is basically identical, and we do not include it here. However, we will need this fact. The only difference from Theorem \ref{jumptimesestimate} is that we allow for initial and final times $(s,u)$ of the bridge to be different from $(0,1)$.

\begin{prop}\label{marginalestshift}
Let $P$ be a counting process of intensity $\ell$, $x \leq y \in \N$ and $0\leq s<u \leq 1$. 
If  for some $\lambda \in \R$,
\bes
\inf_{t \in [s,u] , x \leq z \leq y-1 } \RINV(t,z) \geq \lambda 
\ees
then for  $1\leq i \leq y-x$ and $t \in [s,u]$:
\be\label{estimateshift}
P^{xy}_{su} (X_t  \geq x+i ) \leq \cB_{y-x,\pi^{s,u}_{\lambda}(t) } (\{i,.,y-x\})
\ee
where 
\bes \pi^{s,u}_{\lambda}( t ) = \frac{\exp(\lambda (t-s) )-1}{ \exp(\lambda (u-s) )-1} \ees
\end{prop}
Given the proposition, we can prove Theorem \ref{thm:lln}
\begin{proof}
Fix $\varepsilon>0$ and $k \geq \frac{2}{\ve} \sup_{t \in \Int} \partial_t \pi_{\lambda}(t)$. We show that:
\bes
A:=\left\{ \sup_{0 \leq j \leq k } \vert \frac{1}{N}X_{\frac{j}{k}} - \pi_{\lambda}(\frac{j}{n}) ) \vert \leq \frac{\varepsilon}{2} \right\} 
\subseteq
\left\{ \sup_{t \in \Int} \vert \frac{1}{N}X_t - \pi_{\lambda}( t ) \vert \leq \varepsilon \right\}:=B
\ees
Indeed, assume that $X_{\cdot} \in A$ and take $t \in \Int$. Then $t \in [\frac{j}{k},\frac{j+1}{k}]$ for some $0 \leq j \leq k-1$. Because of the fact that both $\frac{1}{N}X_{\cdot}$ and $\pi_{\lambda}( \cdot)$ are non decreasing functions we have:
\bes
\vert \frac{1}{N}X_t - \pi_{\lambda}(t) \vert \leq \max \left\{ \frac{1}{N}X_{\frac{j+1}{k}} - \pi_{\lambda}(\frac{j}{k}) ,  \pi_{\lambda}(\frac{j+1}{k})-\frac{1}{N}X_{\frac{j}{k}}  \right\} 
\ees
Since $X_{\cdot} \in A$ we have:
\bes
\frac{1}{N}X_{\frac{j+1}{k}} - \pi_{\lambda}(\frac{j}{k}) = \frac{1}{N}X_{\frac{j+1}{k}}  - \pi_{\lambda}(\frac{j+1}{k}) +\pi_{\lambda}(\frac{j+1}{k}) - \pi_{\lambda}(\frac{j}{k}) \leq \frac{\ve}{2} + \frac{1}{k} \sup_{t} \partial_t \pi(\lambda , t) \leq \ve
\ees
where the last inequality is justified by the choice of $k$. In the same way one shows that $  \pi_{\lambda}(\frac{j+1}{k})-\frac{1}{N}X_{\frac{j}{k}}  \leq \ve $.  Since $t$ can be chosen arbitrarily in $\Int$, we have shown that $A \subseteq B$. Because of this, to prove the theorem, it is sufficient to prove that for any finite collection $(t_1,..,t_{k})$ and any $\ve>0$ we have:
\bes
\lim_{N \rightarrow + \infty} P^{0,N}(\sup_{1 \leq l \leq k} \vert \frac{1}{N}X_{t_l} - \pi_{\lambda}(t_l) \vert \geq  \ve ) = 0
\ees
Moreover, since we want to prove convergence to a deterministic limit, it suffices to consider one-time marginals. All we need to show is that for all $t \in \Int $, $\ve \geq 0$
\bes
\lim_{N \rightarrow + \infty} P^{0,N} \left( | \frac{1}{N}X_t - \pi_{\lambda}(t) |  \geq  \ve \right) = 0
\ees
Fix now $t\in \Int,\ve>0$,  and choose $\delta>0$ small enough such that 
\be\label{eq:deltacond}
\pi_{\lambda}(t) + \ve \geq \pi(\lambda-\delta, t ) 
\ee
Because of the assumption \eqref{limitass}, there exist $k$ such that for all 
\be\label{choiceofk} \forall z \geq k \quad \sup_{s \in \Int} \vert \RINV(s,z) - \lambda\vert \leq \delta. \ee
 We choose one such $k$. Let us observe that, if $N$ is large enough  we have
 \be\label{evincl} 
\left\{ \frac{1}{N} X_t - \pi_{\lambda}(t) \geq \varepsilon \right\} \subseteq \left\{ T_{k} \leq t \right\}
 \ee
Using the Markov property we have, by conditioning on $T_k$ and using \eqref{evincl}:
\beas
P^{0,N} \left(  \frac{1}{N}X_t - \pi_{\lambda}(t) \geq \ve \right) &=&  E_{P^{0,N}} \left(E_{P}^{0,N} \left(  \mathbf{1}_{\{ \frac{1}{N}X_t - \pi_{\lambda}(t) \geq \ve \} } \Big \vert T_{k} \right) \mathbf{1}_{\{T_k \leq t \} } \right) \\
&=&  E_{P^{0,N}} \left(P^{k,N}_{T_k,1} \left(  \frac{1}{N}X_t - \pi_{\lambda}(t) \geq \ve  \right)\mathbf{1}_{\{T_k \leq t\}}  \right)
\\
&\leq & \sup_{s \in [0,t]}  P^{k,N}_{s,1} \left(  \frac{1}{N}X_t - \pi_{\lambda}(t) \geq \ve  \right) 
\eeas
Thanks to our choice of $k$, see \eqref{choiceofk}, we are entitled to apply Proposition \ref{marginalestshift}. We get:
\bes
\sup_{s \in [0,t]}  P^{k,N}_{s,1} \left(  \frac{1}{N}X_t - \pi_{\lambda}(t) \geq \ve  \right)  \leq \sup_{s \in [0,t] }\cB_{N-k,\pi^{s,1}_{\lambda-\delta}(t)} ( \{N(\underbrace{\pi_{\lambda}(t)+\ve-k/N}_{:=\xi_n} ) ,.., N-k\}) 
\ees
Since $\pi^{s,1}_{\lambda-\delta}(t)$ is decreasing in $s$, we have:
\bes
\sup_{s \in \Int }\cB_{N-k,\pi^{s,1}_{\lambda - \delta}(t)} ( N \xi_n , N-k ) =  \cB_{N-k,\pi_{\lambda-\delta}(t)} ( N\xi_n , N-k )
\ees
With an application of the Law of large numbers, we have that, as $N$ goes to infinity:
\bes
 \cB_{N-k,\pi_{\lambda-\delta}(t)} (  N A ) \rightarrow \delta_{\pi_{\lambda-\delta}(t)}(A), \quad \forall A\subseteq \Int
\ees
Therefore $ \lim_{N \rightarrow + \infty } \cB_{N-k,\pi_{\lambda-\delta}(t)} ( N\xi_n , N-k ) =0$ if \ $\liminf_{N \rightarrow + \infty} \xi_N > \pi_{\lambda-\delta}(t) $. But this is ensured by our choice of $\delta$, see \eqref{eq:deltacond}. Therefore we have proven that $\lim_{N \rightarrow + \infty} P^{0,N}( \frac{1}{N}X_t - \pi_{\lambda}(t)  \geq  \ve ) = 0$. With a similar argument one shows that $\lim_{N \rightarrow + \infty} P^{0,N}( \frac{1}{N}X_t - \pi_{\lambda}(t)  \leq \ve ) = 0$, from which the  conclusion follows.
\end{proof}

\bibliographystyle{plain}
\bibliography{Ref}
\end{document}